\documentclass[11pt]{article}
\usepackage[T1]{fontenc}
\usepackage[utf8]{inputenc}
\usepackage{lmodern}
\usepackage{amsmath,amssymb,amsthm,mathtools}
\usepackage[margin=1in]{geometry}
\usepackage{microtype,enumitem,xcolor}
\usepackage[colorlinks=true,linkcolor=blue!40!black,citecolor=blue!40!black,urlcolor=blue!40!black]{hyperref}

\newtheorem{theorem}{Theorem}[section]
\newtheorem{lemma}[theorem]{Lemma}
\newtheorem{proposition}[theorem]{Proposition}
\newtheorem{corollary}[theorem]{Corollary}
\theoremstyle{remark}
\newtheorem{remark}[theorem]{Remark}
\newtheorem{question}[theorem]{Question}

\newcommand{\Z}{\mathbb Z}
\newcommand{\floor}[1]{\left\lfloor #1\right\rfloor}
\newcommand{\sq}{\square}
\DeclareMathOperator{\lcm}{lcm}
\numberwithin{equation}{section}
\newcommand{\step}[1]{\smallskip\noindent\emph{#1}\quad}

\hypersetup{pdftitle={Square products of factorials and a conjecture of Erd\H{o}s and Graham},pdfauthor={Fedir Yudin},pdfsubject={Asymptotics and orders of growth for square products of factorials}}

\title{Square products of factorials\\ and a conjecture of Erd\H{o}s and Graham}
\author{Fedir Yudin\\
\small Massachusetts Institute of Technology\\
\small\href{mailto:fedir@mit.edu}{\texttt{fedir@mit.edu}}}
\date{October 1, 2026}

\begin{document}
\maketitle

\begin{abstract}
For $n\ge2$ let $F(n)$ be the least $k\ge2$ such that $n!$ is the largest factor in a product of $k$ distinct factorials that is a perfect square, and let $D_k(X)$ be the number of $n\le X$ with $F(n)=k$. Erd\H{o}s and Graham asked for the order of growth of $D_k(X)$ for $3\le k\le6$, and conjectured that $D_6(X)\gg X$. We prove that
$D_3(X)=\kappa_3\sqrt X+O_\varepsilon(X^{2/5+\varepsilon})$ with an explicit constant $\kappa_3=2.7097\ldots$, and that $D_5(X)\asymp D_6(X)\asymp X$. Together with classical facts, this determines the order of growth of $D_k(X)$ for every $k$. The exponent $2/5$ comes from balancing a uniform bound for Pell equations against Gallagher's larger sieve, with residue restrictions supplied by the Weil bound. For five and six factors we restrict to integers with a prime factor exceeding $X^{1-\alpha}$, where $\alpha>0$ is small and fixed. We exclude shorter representations by combining an equidistribution estimate for primes of Matom\"aki, Radziwi{\l}{\l}, Shao, Tao and Ter\"av\"ainen with the large sieve. In an appendix we use a zero-sum theorem for finite abelian groups to construct, for every $m\ge2$, perfect $m$-th powers that are products of a bounded number of distinct factorials with arguments given by fixed affine functions.
\end{abstract}


\section{Introduction}

For an integer $n\ge2$, let $F(n)$ be the least $k\ge2$ for which there are integers
\[
 1\le a_1<a_2<\cdots<a_k=n\qquad\text{with}\qquad a_1!\,a_2!\cdots a_k!=\sq,
\]
where $\sq$ denotes the square of an integer; put $F(n)=\infty$ if there is no such $k$. (Allowing $a_1=1$ does not change $F(n)$.) For $k\ge2$ let
\[
 \mathcal D_k=\{n\ge2:F(n)=k\},\qquad D_k(X)=\#(\mathcal D_k\cap[1,X]).
\]

These sets were introduced by Erd\H{o}s and Graham~\cite{EG}. Primes admit no representation at all, $\mathcal D_2$ consists of the squares $n\ge4$, and every composite $n$ has $F(n)\le6$, so $\mathcal D_k=\emptyset$ for $k\ge7$.

Erd\H{o}s and Graham showed that $D_3(X)=o(X)$, observed that $D_4(X)\gg X$, and wrote~\cite[p.~354]{EG}: ``We are reasonably certain that $D_6(n)>cn$.'' The question of the growth of $D_k(X)$ for $3\le k\le6$ is Problem~374 in Bloom's list of Erd\H{o}s problems~\cite{Bloom}.

For three factorials, Luca, Saradha and Shorey~\cite[Theorem~1]{LSS} proved
\[
 D_3(X)\ll X\exp\bigl(-c_0(\log X)^{1/4}(\log\log X)^{3/4}\bigr)
\]
for some $c_0>0$, and Tao~\cite[Theorem~1.10]{T} recently showed that the equation $a_1!a_2!a_3!=\sq$ has $X^{1/2+o(1)}$ solutions with $1\le a_1<a_2<a_3\le X$. The solutions with $a_2=a_3-1$ are easy to count: $n!(n-1)!c!=\sq$ holds precisely when $n$ and $c!$ have the same squarefree part. Write $s(v)$ for the squarefree part of $v$, and put
\begin{equation}\label{eq:kappa}
 \mathcal S=\{s(c!):c\ge1\},\qquad
 \kappa_3=\sum_{a\in\mathcal S\setminus\{1\}}\frac1{\sqrt a}=2.709751\ldots,
\end{equation}
each element of $\mathcal S$ being counted once. The series converges; see Section~\ref{subsec:main-term}. Our first result shows that these consecutive representations account for the main term of $D_3(X)$.

\begin{theorem}\label{thm:three}
For every $\varepsilon>0$,
\[
 D_3(X)=\kappa_3\sqrt X+O_\varepsilon\bigl(X^{2/5+\varepsilon}\bigr).
\]
\end{theorem}

In the notation of~\cite[Section~1.3, (1.17)]{T} we have $\kappa_3=c_3^1-1$; the subtraction of $1$ accounts for the squares, which lie in $\mathcal D_2$. The new estimate is Theorem~\ref{thm:nonconsecutive}: only $O_\varepsilon(X^{2/5+\varepsilon})$ integers $n\le X$ admit a representation $n!(n-\ell)!c!=\sq$ with $\ell\ge2$.

Our second result confirms the conjecture of Erd\H{o}s and Graham, together with its analogue for five factorials.

\begin{theorem}\label{thm:main}
There is a constant $c_*>0$ such that, for $k\in\{5,6\}$ and all sufficiently large $X$, at least $c_*X$ squarefree integers $n\in[X/2,X]$ satisfy $F(n)=k$. For $k=5$ this holds separately for even and for odd $n$.
\end{theorem}

Combining these results with the classical facts above gives the order of growth of every $D_k$.

\begin{corollary}\label{cor:growth}
As $X\to\infty$,
\[
\begin{gathered}
 D_2(X)=\floor{\sqrt X}-1,\qquad D_3(X)\sim\kappa_3\sqrt X,\\
 D_4(X)\asymp D_5(X)\asymp D_6(X)\asymp X.
\end{gathered}
\]
For $k\ge7$, we have $D_k(X)=0$.
\end{corollary}

\begin{proof}
The statements for $k=2$ and $k\ge7$ were recalled above; see also Section~\ref{subsec:main-term}. If $n=v^2u$ with $v\ge2$ and $u\ge2$, then
\begin{equation}\label{eq:nonsquarefree}
 n!\,(n-1)!\,u!\,(u-1)!=\bigl(vu\,(n-1)!\,(u-1)!\bigr)^2,
\end{equation}
with four distinct arguments, so $F(n)\le4$. By Theorem~\ref{thm:three}, all but $O(\sqrt X)$ nonsquarefree $n\le X$ therefore lie in $\mathcal D_4$, which gives $D_4(X)\ge(1-6/\pi^2+o(1))X$. The remaining assertions are Theorems~\ref{thm:three} and~\ref{thm:main}.
\end{proof}

Corollary~\ref{cor:growth} settles the orders of growth asked for in Problem~374. Whether $D_k(X)/X$ tends to a limit for $k=4,5,6$ remains open; see Section~\ref{sec:questions}.

\subsection*{Outline of the proofs}
Section~\ref{sec:three} is independent of the rest of the paper; the proof of Theorem~\ref{thm:main} uses only the bound~\eqref{eq:all-three} that $O(\sqrt X)$ integers have a three-factor representation.

\emph{Three factorials.} Write a representation as $n!(n-\ell)!c!=\sq$ with $n>n-\ell>c\ge1$. Removing the square $((n-\ell)!)^2$ leaves
\[
 c!\,n(n-1)\cdots(n-\ell+1)=\sq.
\]
For $\ell=1$ this is the condition $s(n)=s(c!)$, which produces the main term. For $\ell\ge2$, Tao~\cite{T} showed that $\ell$ is smaller than any fixed positive power of $n$, and then so is $c$. If $s(c!)$ is small, two terms of the block $n,n-1,\dots,n-\ell+1$ have squarefree parts with small product, and a uniform bound for Pell equations applies. If $s(c!)$ is large, its prime factors exceeding $\ell$ restrict $n$ to few residue classes; in addition, at larger primes $q$ the product must be a square modulo $q$, which by the Weil bound excludes about half of the residues. Gallagher's larger sieve combines these two kinds of information. Balancing the two cases gives the exponent $2/5$.

\emph{Five and six factorials.} We consider $n=ap$ where $p>X^{1-\alpha}$ is prime and $a<X^\alpha$, for a small fixed $\alpha>0$. The identity
\begin{equation}\label{eq:six}
 n!\,(n-1)!\,a!\,(a-1)!\,p!\,(p-1)!=\bigl(n\,(n-1)!\,(a-1)!\,(p-1)!\bigr)^2
\end{equation}
of Erd\H{o}s and Graham shows that $F(n)\le6$, so we must rule out shorter representations for many such $n$. By~\eqref{eq:nonsquarefree}, $n$ must be squarefree. Pairing the four largest factorials in a representation with arguments $n>n-\ell_1>m>m-\ell_2>c$ leads to
\[
 n(n-1)\cdots(n-\ell_1+1)\cdot m(m-1)\cdots(m-\ell_2+1)\cdot c!=\sq.
\]
The preliminary bounds put both block lengths and $c$ below $p$. Since $p>\sqrt n$ divides $n$ exactly once, it must then divide a term $m-j$ of the second block, so $m-j=(t/a)n$ for some $1\le t<a$. We use this relation in four steps.
\begin{enumerate}[label=(\arabic*)]
\item Both blocks are shorter than $X^\varepsilon$ (Lemma~\ref{lem:short}). Otherwise we find a prime $q$ that divides exactly one term of one block, to the first power, which is incompatible with the square condition. We construct a smooth weight that separates the two blocks. The relation $m-j=(t/a)n$ gives this weight a positive average over each period of length at most $a$. After retaining complete periods on which square divisibility is excluded, an equidistribution theorem of Matom\"aki, Radziwi{\l}{\l}, Shao, Tao and Ter\"av\"ainen~\cite{MRST} gives the required prime.
\item If $s(n-1)$ has a prime factor larger than $X^{3\varepsilon}$, the first block has length one (Lemma~\ref{lem:transfer}).
\item If the first block has length one, the second block has length at least two for $O(X^{1/2+5\varepsilon}(\log X)^2)$ values of $n$ (Lemma~\ref{lem:count}). This follows from the Weil bound and the large sieve.
\item If both blocks have length one, the square condition reduces to $atc!=\sq$ with $t<a$. For five factors this is excluded by a condition on $a$ alone, and for four factors it is impossible for squarefree $n$.
\end{enumerate}
Finally we show that the integers $n=ap$ satisfying all these conditions have positive lower density (Proposition~\ref{prop:candidates}); for them~\eqref{eq:six} shows that $F(n)=6$. Choosing instead $a$ with prescribed divisibility properties and using identities due to Straus gives $F(n)=5$.

\emph{Perfect powers.} At the end of~\cite{EG}, Erd\H{o}s and Graham remark that for every $m$ there are infinitely many $m$-th powers that are products of a bounded number of factorials. In Appendix~\ref{app:affine} we give a construction with distinct factorials whose arguments are fixed affine functions of a parameter, using a zero-sum theorem for abelian groups.

\section{Notation and tools}\label{sec:tools}

For a positive integer $v$, $s(v)$ denotes its squarefree part, so that $v/s(v)$ is a square, and $P(v)$ its largest prime factor, with $P(1)=1$. We write $\tau(v)$ for the number of divisors of $v$. For $\ell\ge1$ we use the falling factorial
\[
 (x)_\ell=x(x-1)\cdots(x-\ell+1).
\]
As usual, $\pi(x)$ is the number of primes up to $x$ and $\vartheta(x)=\sum_{p\le x}\log p$. The letters $p$ and $q$ always denote primes. All logarithms are natural, and implied constants are absolute unless their dependence is indicated.

\subsection{Primes}
Besides the prime number theorem and the estimate $\sum_{p\le y}(\log p)/p=\log y+O(1)$, we use Huxley's theorem on primes in short intervals~\cite{Hux} in the following weak form: uniformly for $x^{3/5}\le y\le x/2$,
\begin{equation}\label{eq:huxley}
 \vartheta(x)-\vartheta(x-y)\sim y\qquad(x\to\infty).
\end{equation}

\subsection{Pell equations}
We use the following uniform bound, in the form stated by Tao~\cite[Lemma~2.10]{T}.

\begin{lemma}\label{lem:pell}
Fix $K>0$. For positive integers $e_1,e_2\le X^K$ and a nonzero integer $\Delta$ with $|\Delta|\le X^K$, the equation
\[
 e_1 u^2-e_2 v^2=\Delta
\]
has $X^{o(1)}$ solutions in positive integers $u\le X$ and $v$, uniformly in $e_1,e_2,\Delta$.
\end{lemma}

\subsection{Sieves}
We use the large sieve in the following form; see~\cite[Corollary~2.8]{T}. Let $\mathcal P$ be a nonempty set of primes $q\le\sqrt X$, and suppose that $E\subseteq[1,X]\cap\Z$ lies in at most $\rho(q)$ residue classes modulo $q$ for each $q\in\mathcal P$, where $1\le\rho(q)<q$. Then
\begin{equation}\label{eq:sieve}
 |E|\ll X\Biggl(\sum_{q\in\mathcal P}\frac{q-\rho(q)}{\rho(q)}\Biggr)^{-1}.
\end{equation}
We also use Gallagher's larger sieve~\cite{G} in Section~\ref{sec:three}.

\begin{lemma}[Gallagher]\label{lem:larger}
Let $A\subseteq[1,X]\cap\Z$ with $X>1$. Let $\mathcal P$ be a finite set of primes, and suppose that $A$ lies in at most $\rho(p)$ residue classes modulo $p$ for each $p\in\mathcal P$, where $1\le\rho(p)\le p$. If
\[
 \Delta=\sum_{p\in\mathcal P}\frac{\log p}{\rho(p)}-\log X>0,
\]
then
\[
 |A|\le\frac1\Delta\Bigl(\sum_{p\in\mathcal P}\log p-\log X\Bigr).
\]
\end{lemma}

\begin{proof}
We may assume $N=|A|\ge1$. For $p\in\mathcal P$ let $N_{p,r}$ be the number of elements of $A$ in the class $r$ modulo $p$. Since $\sum_rN_{p,r}=N$ and at most $\rho(p)$ of these counts are nonzero, Cauchy--Schwarz gives $\sum_rN_{p,r}^2\ge N^2/\rho(p)$. The quantity $\sum_rN_{p,r}^2-N$ counts ordered pairs of distinct elements of $A$ whose difference is divisible by $p$. For each such pair the primes $p\in\mathcal P$ dividing the difference have $\sum\log p\le\log X$. Hence
\[
 \sum_{p\in\mathcal P}\log p\Bigl(\frac{N^2}{\rho(p)}-N\Bigr)\le N(N-1)\log X,
\]
and dividing by $N$ gives the lemma.
\end{proof}

\subsection{Character sums}
For an odd prime $q$ let $\chi_q$ be the quadratic character modulo $q$: it is $0$ at zero, $1$ at a nonzero square, and $-1$ at a nonsquare. We use the Weil bound in the following form~\cite[Theorem~5.41]{LN}: if $f=\lambda\prod_{i=1}^{\ell}(z-r_i)\in\mathbb F_q[z]$ with $\lambda\ne0$ and distinct $r_1,\dots,r_\ell\in\mathbb F_q$, then
\begin{equation}\label{eq:weil}
 \Bigl|\sum_{z\in\mathbb F_q}\chi_q(f(z))\Bigr|\le(\ell-1)\sqrt q.
\end{equation}

\section{Three factorials}\label{sec:three}

\subsection{Representations with nonconsecutive largest arguments}
Let
\[
 E(X)=\#\bigl\{n\le X:\ n!\,(n-\ell)!\,c!=\sq\ \text{for some}\ \ell\ge2\ \text{and}\ 1\le c<n-\ell\bigr\}.
\]

\begin{theorem}\label{thm:nonconsecutive}
For every $\varepsilon>0$ we have $E(X)\ll_\varepsilon X^{2/5+\varepsilon}$.
\end{theorem}

Removing the square $((n-\ell)!)^2$, a representation counted by $E(X)$ satisfies
\begin{equation}\label{eq:three-block}
 c!\,(n)_\ell=\sq.
\end{equation}
In Tao's terminology~\cite[Definition~1.2(iii)]{T}, this says that $\{N+1,\dots,N+H\}$ with $N=n-\ell$ and $H=\ell$ is an interval of type $F_3$, the relevant factorial being $c!$ with $1\le c<N$. By~\cite[Lemma~4.2]{T}, every interval of type $F_3$ satisfies
\[
 H\le\exp\bigl((\log N)^{2/3+o(1)}\bigr)\qquad(N\to\infty),
\]
and by~\cite[Lemma~4.1]{T} we have $H<N$, so $N>n/2$. The bound on the smallest argument is elementary, as in Lemma~\ref{lem:smallest} below: every prime in $(c/2,c]$ occurs exactly once in $c!$, and hence divides $(n)_\ell$, so that $\vartheta(c)-\vartheta(c/2)\le\ell\log n$. Consequently
\begin{equation}\label{eq:three-short}
 \ell\le\exp\bigl((\log n)^{2/3+o(1)}\bigr),\qquad c\ll\ell\log n.
\end{equation}
In particular, for every fixed $\eta>0$ and every sufficiently large $X$, all such representations with $n\le X$ satisfy $\ell,c\le X^\eta$; the finitely many exceptions with small $n$ do not affect Theorem~\ref{thm:nonconsecutive}.

\begin{proof}[Proof of Theorem~\ref{thm:nonconsecutive}]
Fix $0<\eta<1/100$ and let $X$ be large, so that $\ell,c\le X^\eta$ in every representation under consideration.

\step{Case 1: $s(c!)\le X^{\ell/5}$.}
Write $n-i=d_iu_i^2$ with $d_i=s(n-i)$ for $0\le i<\ell$. A prime $p>\ell$ divides at most one term of the block, so by~\eqref{eq:three-block} it divides one of the $d_i$ if and only if it divides $s(c!)$. A prime $p\le\ell$ divides at most $\ell/p+1$ terms. Hence
\[
 \log\prod_{i=0}^{\ell-1}d_i\le\log s(c!)+\sum_{p\le\ell}\Bigl(\frac\ell p+1\Bigr)\log p=\log s(c!)+\ell\log\ell+O(\ell).
\]
Choose $i<j$ so that $d_i,d_j$ are the two smallest squarefree parts. Then
\[
 d_id_j\le\Bigl(\prod_{k=0}^{\ell-1}d_k\Bigr)^{2/\ell}\ll\ell^2s(c!)^{2/\ell}\ll X^{2/5+2\eta}.
\]
Subtracting the corresponding terms gives
\[
 d_iu_i^2-d_ju_j^2=j-i\ne0.
\]
We count these Pell equations without fixing $c$ or $\ell$. There are $X^{2/5+2\eta+o(1)}$ pairs of positive integers $(d_i,d_j)$ with product $O(X^{2/5+2\eta})$, and at most $X^\eta$ values of $j-i$. For each choice, Lemma~\ref{lem:pell} gives $X^{o(1)}$ solutions $(u_i,u_j)$, and $n=d_iu_i^2+i$ is then determined up to at most $X^\eta$ choices of $i$. Case~1 therefore contributes at most $X^{2/5+4\eta+o(1)}$ values of $n$.

\step{Case 2: $s(c!)>X^{\ell/5}$.}
Fix $c$ and $\ell$, and let $A_{c,\ell}$ be the set of possible integers $n\le X$ for this pair. We apply Lemma~\ref{lem:larger} with two sets of primes.

First let $\mathcal P_0$ be the set of primes $p>\ell$ dividing $s(c!)$. Such a prime occurs to an odd power in $c!$, so it must divide a term of the block; since $p>\ell$, this confines $n$ to the $\ell$ classes $0,1,\dots,\ell-1$ modulo $p$. These primes contribute
\[
 \sum_{p\in\mathcal P_0}\frac{\log p}{\ell}\ge\frac{\log s(c!)-\vartheta(\ell)}{\ell}\ge\frac{\log X}5-O(1).
\]

Next let $\mathcal P_1$ be the set of primes $q$ with $Y<q\le Q$, where
\[
 Y=X^{3\eta},\qquad Q=X^{2/5+4\eta}.
\]
These primes exceed $c$ and $\ell$; in particular $\mathcal P_1$ is disjoint from $\mathcal P_0\subseteq[1,c]$. For $q\in\mathcal P_1$, every $n\in A_{c,\ell}$ makes $f(n)$ a square modulo $q$, where $f(z)=c!\,(z)_\ell$ has the $\ell$ distinct roots $0,1,\dots,\ell-1$ and a nonzero leading coefficient modulo $q$. The number of residues $z$ for which $f(z)$ is a square or zero is
\[
 \rho(q)=\frac12\Bigl(q+\ell+\sum_{z\in\mathbb F_q}\chi_q(f(z))\Bigr)\le\frac q2+\frac{\ell+(\ell-1)\sqrt q}2
\]
by~\eqref{eq:weil}. As $\ell/\sqrt q\le X^{-\eta/2}$, this gives $1/\rho(q)\ge(2/q)(1-O(X^{-\eta/2}))$, and so
\[
 \sum_{q\in\mathcal P_1}\frac{\log q}{\rho(q)}\ge(2-o(1))\log\frac QY+O(1)=\Bigl(\frac45+2\eta-o(1)\Bigr)\log X.
\]

Adding the two contributions and subtracting $\log X$, the denominator $\Delta$ in Lemma~\ref{lem:larger} is at least $\eta\log X$ for large $X$. The numerator is at most $\vartheta(c)+\vartheta(Q)\ll Q$. Hence
\[
 |A_{c,\ell}|\ll_\eta\frac{X^{2/5+4\eta}}{\log X}.
\]
Summing over the at most $X^{2\eta}$ pairs $(c,\ell)$, Case~2 contributes $O_\eta(X^{2/5+6\eta})$ values of $n$.

Combining the two cases and taking $\eta$ small in terms of $\varepsilon$ proves the theorem.
\end{proof}

\begin{remark}
The exponent $2/5$ comes from balancing the two cases. If the cases are separated at $s(c!)=X^{\sigma\ell}$, the first contributes $X^{2\sigma+o(1)}$, while in the second the primes dividing $s(c!)$ supply $\sigma\log X$ of the sieve denominator and the quadratic residue conditions must supply the rest, which leads to the bound $X^{(1-\sigma)/2+o(1)}$. The two exponents agree when $\sigma=1/5$. Without the first set of primes the larger sieve alone would give only $X^{1/2+o(1)}$.
\end{remark}

\subsection{The main term}\label{subsec:main-term}

\begin{proof}[Proof of Theorem~\ref{thm:three}]
For $a\in\mathcal S$ let $c(a)$ be the least $c\ge1$ with $s(c!)=a$. A consecutive representation $n!(n-1)!c!=\sq$ exists if and only if $n=au^2$ for some $a\in\mathcal S$ with $c(a)<n-1$. The families $\{au^2\}$ are disjoint, since $a=s(n)$.

Every prime in $(c/2,c]$ divides $s(c!)$, so $s(c!)\ge e^{\delta c}$ for some $\delta>0$ and all large $c$. Hence
\begin{equation}\label{eq:S-count}
 \#\{a\in\mathcal S:a\le Z\}\ll\log(2Z),\qquad\sum_{a\in\mathcal S}a^{-1/2}<\infty.
\end{equation}
The same estimate gives $c(a)\ll\log(2a)$, so $c(a)<a-1$ for all but finitely many $a$, and the condition $c(a)<n-1$ excludes only $O(1)$ integers $n$. The number of $n\le X$ of this form is therefore
\[
 \sum_{\substack{a\in\mathcal S\\ a\le X}}\floor{\sqrt{X/a}}+O(1)=(\kappa_3+1)\sqrt X+O(\log X),
\]
since the rounding errors are $O(\log X)$ by~\eqref{eq:S-count}, and so is the tail:
\[
 \sqrt X\sum_{\substack{a\in\mathcal S\\ a>X}}a^{-1/2}\ll\sum_{j\ge0}2^{-j/2}\log(2^{j+1}X)\ll\log X.
\]
Every other $n$ admitting a three-factor representation is counted by $E(X)$. By Theorem~\ref{thm:nonconsecutive}, the number of $n\le X$ admitting a three-factor representation is therefore
\begin{equation}\label{eq:all-three}
 (\kappa_3+1)\sqrt X+O_\varepsilon(X^{2/5+\varepsilon}).
\end{equation}

It remains to remove the $n$ with $F(n)=2$. By the theorem of Erd\H{o}s and Selfridge~\cite[Theorem~1]{ES}, a product of two or more consecutive positive integers is never a perfect power. Hence $n!r!=\sq$ with $r<n$ forces $r=n-1$ and $n=\sq$; that is, $\mathcal D_2$ consists of the squares $n\ge4$. These squares were included in~\eqref{eq:all-three} (take $c=1$), and there are $\sqrt X+O(1)$ of them up to $X$. Subtracting them proves Theorem~\ref{thm:three}.
\end{proof}

\section{Four and five factorials: two blocks}\label{sec:two-blocks}

To prove that many integers require five or six factors, we must exclude shorter representations. The integers with a three-factor representation have already been counted in~\eqref{eq:all-three}; for squarefree integers, this leaves representations with four or five factorials to consider. Pairing the four largest factorials reduces the problem to two blocks of consecutive integers. Write the arguments of a five-factor representation as
\[
 n>n-\ell_1>m>m-\ell_2>c\ge1.
\]
Since $n!(n-\ell_1)!=(n)_{\ell_1}((n-\ell_1)!)^2$ and $m!(m-\ell_2)!=(m)_{\ell_2}((m-\ell_2)!)^2$, the representation is equivalent to
\begin{equation}\label{eq:two-block}
 (n)_{\ell_1}\,(m)_{\ell_2}\,c!=\sq,\qquad n>n-\ell_1>m>m-\ell_2>c\ge0.
\end{equation}
A four-factor representation corresponds to~\eqref{eq:two-block} with $c=0$. We call
\[
 B_1=\{n-\ell_1+1,\dots,n\},\qquad B_2=\{m-\ell_2+1,\dots,m\}
\]
the \emph{upper} and \emph{lower block}; they have lengths $\ell_1,\ell_2\ge1$. When $\ell_1=\ell_2=1$ the four largest arguments form two consecutive pairs.

\begin{lemma}\label{lem:smallest}
Every solution of~\eqref{eq:two-block} satisfies $c\ll(\ell_1+\ell_2)\log n+1$.
\end{lemma}

\begin{proof}
Every prime in $(c/2,c]$ occurs exactly once in $c!$ and therefore divides $(n)_{\ell_1}(m)_{\ell_2}$. Hence $\vartheta(c)-\vartheta(c/2)\le(\ell_1+\ell_2)\log n$, and the claim follows from the prime number theorem.
\end{proof}

\begin{lemma}\label{lem:coarse}
Every solution of~\eqref{eq:two-block} satisfies
\[
 \ell_1\ll n^{3/5},\qquad\ell_2\ll\ell_1\log n+n^{3/5}.
\]
In particular $L=\max(\ell_1,\ell_2)\le n^{2/3}$ for all sufficiently large $n$.
\end{lemma}

\begin{proof}
If $\ell_1\ge n/2$, a prime in $(n/2,n]$ (which exists by Bertrand's postulate) divides exactly one term of $B_1$ and nothing else in~\eqref{eq:two-block}, which is impossible. Hence $\ell_1<n/2$, and the same argument shows that $(n-\ell_1,n]$ contains no prime. By~\eqref{eq:huxley}, $\ell_1\ll n^{3/5}$.

For the lower block we may assume $\ell_2\ge m^{3/5}$ and $m$ large. Put $y=\min(\ell_2,m/2)$, so that $m^{3/5}\le y\le m/2$ and $\ell_2\le2y$. Each prime in $(m-y,m]$ divides exactly one term of $B_2$, does not divide $c!$, and therefore divides $(n)_{\ell_1}$. By~\eqref{eq:huxley},
\[
 \ell_2\le2y\ll\vartheta(m)-\vartheta(m-y)\le\log((n)_{\ell_1})\le\ell_1\log n.\qedhere
\]
\end{proof}

To force the upper block to have length one, we will choose $n$ so that $s(n-1)$ has a large prime factor. The next lemma gives the restriction on this prime factor when $\ell_1\ge2$.

\begin{lemma}\label{lem:transfer}
Let $n=ap$ with $p$ prime, $p>\sqrt n$ and $p>\max(\ell_1,\ell_2,c)$, and suppose that~\eqref{eq:two-block} holds with $\ell_1\ge2$. Then
\[
 P\bigl(s(n-1)\bigr)\le\max(\ell_1,c,a\ell_2).
\]
\end{lemma}

\begin{proof}
The prime $p$ divides $n$ exactly once, divides no other term of $B_1$, and does not divide $c!$. It must therefore divide a term of $B_2$. Since $p>\ell_2$, that term is unique:
\[
 m-j=tp\qquad\text{for some }0\le j<\ell_2\text{ and }1\le t<a,
\]
where $t<a$ because $m<n$.

Now let $q>\max(\ell_1,c,a\ell_2)$ be a prime dividing $n-1$ to an odd power. It divides no other term of $B_1$ and does not divide $c!$. If it divided a term $m-i$ of $B_2$, it would divide
\[
 a(m-i)-t(n-1)=a(j-i)+t.
\]
This integer is nonzero since $0<t<a$, and has absolute value less than $a\ell_2<q$, which is a contradiction. So $q$ occurs to an odd power in~\eqref{eq:two-block}, which is impossible.
\end{proof}

\section{Long blocks}\label{sec:long}

\begin{lemma}\label{lem:short}
Let $0<\alpha<\varepsilon\le10^{-3}$. If $n$ is sufficiently large, $P(n)>n^{1-\alpha}$, and~\eqref{eq:two-block} holds, then $\max(\ell_1,\ell_2)<n^\varepsilon$.
\end{lemma}

We prove the lemma by finding a prime that divides exactly one term of the two blocks, to the first power, and does not divide $c!$. Such a prime contradicts the square condition in~\eqref{eq:two-block}. The following estimate will supply the prime once we have constructed a smooth weight with a positive integral.

\begin{lemma}\label{lem:prime-comparison}
Fix $\varepsilon>0$. Let $n$ be large, $n^\varepsilon\le L\le n^{2/3}$, $Q=L(\log n)^2$, $0<\theta<1$ and $M\in\{n,\theta n\}$. Let $W:\mathbb R^3\to[0,1]$ be smooth and $\Z^3$-periodic, with $\|\partial^\beta W\|_\infty\ll_\beta(\log n)^{C_W|\beta|}$ for some fixed $C_W>0$ and every multi-index $\beta$. Then, for every fixed $\Lambda>0$,
\[
 \sum_{Q<q\le2Q}W\Bigl(\frac nq,\frac{\theta n}q,\frac M{q^2}\Bigr)
 =\int_Q^{2Q}W\Bigl(\frac nv,\frac{\theta n}v,\frac M{v^2}\Bigr)\frac{dv}{\log v}
 +O\Bigl(\frac Q{(\log n)^\Lambda}\Bigr),
\]
where the implied constant depends on $\varepsilon$, $\Lambda$, $C_W$ and the implied constants in the derivative bounds.
\end{lemma}

\begin{proof}
Expand $W$ in a Fourier series. The derivative bounds allow us to truncate at frequencies bounded by a power of $\log n$, with uniform error $O((\log n)^{-B})$ for any prescribed $B>0$. The resulting error in either the prime sum or the integral is $O(Q(\log n)^{-B})$.

Each remaining Fourier mode contributes the phase
\[
 \frac{(k_1+k_2\theta)n}{q}+\frac{k_3M}{q^2}
\]
with integers $k_1,k_2,k_3$ of polylogarithmic size. Proposition~1.12(i) of~\cite{MRST}, with exponent $j=2$, compares the prime sum for each such mode with its integral, with error $O_B(Q(\log Q)^{-B})$ for every fixed $B>0$, provided both real phase coefficients are $O(\exp((\log Q)^{5/4}))$. Our coefficients are $O(n(\log n)^{O(1)})$, and $\log n=o((\log Q)^{5/4})$ since $Q\ge n^\varepsilon$. As there are only $(\log n)^{O(1)}$ modes, with total coefficient size $(\log n)^{O(1)}$, the lemma follows.
\end{proof}

\begin{proof}[Proof of Lemma~\ref{lem:short}]
Suppose that $L=\max(\ell_1,\ell_2)\ge n^\varepsilon$, and write $n=ap$ with $p=P(n)$, so that $a<n^\alpha$. By Lemmas~\ref{lem:smallest} and~\ref{lem:coarse}, $p>\max(\ell_1,\ell_2,c,\sqrt n)$ for large $n$. As in the proof of Lemma~\ref{lem:transfer}, $p$ divides exactly one term $m-j$ of $B_2$, and
\begin{equation}\label{eq:rational}
 m-j=\theta n,\qquad\theta=\frac ta\in(0,1),\qquad 1\le t<a.
\end{equation}
Thus $B_2$ consists of the integers $\theta n+e$ with $j-\ell_2<e\le j$, and in particular $|e|<L$. Write $\theta=A_0/A$ in lowest terms, so $2\le A\le a<n^\alpha$. Put
\[
 Q=L(\log n)^2,\qquad\delta=\frac LQ=(\log n)^{-2},\qquad J=\Bigl[\frac n{2Q},\frac nQ\Bigr].
\]
By Lemma~\ref{lem:smallest}, $c\ll L\log n$, so $Q>\max(\ell_1,\ell_2,c)$ for large $n$. Since $A<n^\alpha$, $n^\varepsilon\le L\le n^{2/3}$ and $\alpha<\varepsilon$,
\begin{equation}\label{eq:scales}
 \frac AQ=o(1),\qquad\frac{AQ}n\ll n^{-1/3+\alpha}(\log n)^2=o(1).
\end{equation}
The second estimate ensures that $J$ contains many periods of length $A$. The first will let us exclude square divisibility while retaining whole periods.

Every prime $q\in(Q,2Q]$ divides at most one term of each block and does not divide $c!$. We first construct a weight detecting primes that divide one block and avoid the other. We then modify it to exclude divisibility by $q^2$.

\step{Separating the blocks.}
We seek a smooth $\Z^2$-periodic function $H:\mathbb R^2\to[0,1]$ such that $H(n/q,\theta n/q)>0$ implies that $q$ divides a term of exactly one block, and such that
\begin{equation}\label{eq:period-mass}
 \int_I H(u,\theta u)\,du\gg A\delta
\end{equation}
for every interval $I=[u_0,u_0+A]$ with $u_0\in\Z$. We also choose $M\in\{n,\theta n\}$ so that every term of the selected block lies within $L$ of $M$.

For smooth $1$-periodic functions $f,g$, we have the identity
\begin{equation}\label{eq:period-average}
 \int_I f(u)g(\theta u)\,du
 =\int_0^1f(x)\sum_{k=0}^{A-1}g\Bigl(\frac kA+\theta x\Bigr)dx.
\end{equation}
Indeed, split $I$ into its $A$ unit intervals. On the interval beginning at $u_0+r$, substitute $u=u_0+r+x$; the residues $\theta(u_0+r)$ run through $k/A$ modulo $1$, since $A_0$ and $A$ are coprime. We will also use $\int_I f(\theta u)\,du=A\int_0^1f$, which follows directly from $\theta A=A_0\in\Z$.

\smallskip\noindent\emph{Case 1: $\ell_1\ge\ell_2/100$.}
Then $\ell_1\ge L/100$. We select the upper block and take $M=n$. Choose smooth $1$-periodic functions $f,g:\mathbb R\to[0,1]$ such that $f$ is supported in $(\ell_1/(16Q),3\ell_1/(8Q))$ modulo $1$ and equals $1$ on $[\ell_1/(8Q),\ell_1/(4Q)]$, and $g$ is supported in $(\delta,1-\delta)$ modulo $1$ and equals $1$ on $[2\delta,1-2\delta]$. Set $H(x,y)=f(x)g(y)$.

If $H(n/q,\theta n/q)>0$, then $0<q\{n/q\}<3\ell_1/4$. Since $q\{n/q\}$ is an integer, $q$ divides a term of $B_1$. Also $\{\theta n/q\}\in(\delta,1-\delta)$, so $\theta n$ is at distance more than $\delta q>L$ from every multiple of $q$. Thus $q$ divides no term of $B_2$.

At least $\floor{A/2}\ge A/3$ of the points $k/A$, $0\le k<A$, lie in $[1/4,3/4]$. For each of them and every $x$ in the support of $f$ in $[0,1)$, we have $0\le\theta x\le3\delta/8$, hence $g(k/A+\theta x)=1$. By~\eqref{eq:period-average},
\[
 \int_I H(u,\theta u)\,du
 \ge\frac A3\int_0^1f\ge\frac{A\ell_1}{24Q}\ge\frac{A\delta}{2400},
\]
which proves~\eqref{eq:period-mass} in this case.

\smallskip\noindent\emph{Case 2: $\ell_2>100\ell_1$.}
Now $L=\ell_2$. We select the lower block and take $M=\theta n$. Let $f$ be a smooth $1$-periodic function with values in $[0,1]$, supported strictly inside
\[
 \Bigl(-\frac j{2Q},\frac{\ell_2-j}{2Q}\Bigr)\pmod1
\]
and equal to $1$ on its middle subinterval of length $\delta/8$. The whole interval has length $\delta/2$, and $\int_0^1f\ge\delta/8$. Let $g$ be a smooth $1$-periodic function with values in $[0,1]$ which vanishes on $[0,\ell_1/Q]$ modulo $1$ and equals $1$ outside
\[
 E=\Bigl[0,\frac{\ell_1}Q+\frac\delta{1000}\Bigr]\cup\Bigl[1-\frac\delta{1000},1\Bigr],
\]
so that $|E|<0.012\,\delta$. Set $H(x,y)=g(x)f(y)$.

If $H(n/q,\theta n/q)>0$, let $\xi$ be the representative of $\theta n/q$ modulo $1$ in the support of $f$. Since $\theta n=tp$ is an integer, so is $q\xi$, and $-j<q\xi<\ell_2-j$ because $q\le2Q$. Hence $\theta n-q\xi$ is a term of $B_2$ divisible by $q$. Since $g(n/q)\ne0$, we have $\{n/q\}>\ell_1/Q>\ell_1/q$, so $q$ divides no term of $B_1$.

On every interval $I$ of length $A$ with integer endpoints,
\[
 \int_I H(u,\theta u)\,du
 \ge\int_I f(\theta u)\,du-\int_I(1-g(u))\,du
 \ge\frac{A\delta}{8}-A|E|\ge\frac{A\delta}{10}.
\]
This proves~\eqref{eq:period-mass} in the second case.

In both cases we construct $f$ and $g$ by translating and rescaling fixed smooth functions. In Case~1, their transition widths can be chosen at least a fixed multiple of $\delta$, since $\ell_1/Q\ge\delta/100$. In Case~2, the transition widths of $f$ are comparable to $\delta$, and those of $g$ are $\delta/1000$, independently of $\ell_1/Q$. Each differentiation costs the inverse of a transition width, so
\[
 \|f^{(r)}\|_\infty+\|g^{(r)}\|_\infty\ll_r\delta^{-r},
\]
uniformly in the parameters; changing $j$ in Case~2 only translates $f$.

\step{Excluding square divisibility.}
The selected prime might still divide its term more than once. We now exclude this possibility while preserving~\eqref{eq:period-mass} on a fixed proportion of the complete periods in $J$.

Every term of the selected block lies within $L$ of $M$. Put $U=M/Q^2$. If $U<1/4$, then $M+L<Q^2$ for large $n$, so no term of the block is divisible by $q^2$ for $q>Q$, and we take $G=1$. Otherwise, fix a small absolute constant $\tau_0>0$ and a smooth $1$-periodic $G:\mathbb R\to[0,1]$ supported where $\tau_0<\{\zeta\}<1-\tau_0$ and equal to $1$ where $2\tau_0\le\{\zeta\}\le1-2\tau_0$. If $G(M/q^2)>0$, then $M$ is at distance more than $\tau_0q^2>L$ from every multiple of $q^2$, so again no term of the block is divisible by $q^2$.

Under the substitution $u=n/v$, the interval $[Q,2Q]$ becomes $J$, and $M/v^2$ becomes $\phi(u)=Mu^2/n^2$. Partition $J$ into intervals $I=[u_0,u_0+A]$ with $u_0\in\Z$, discarding two end pieces of total length at most $2A$. Since $0<\phi'(u)\le2/Q$ on $J$, the variation of $\phi$ on each period is at most $2A/Q<\tau_0$ for large $n$, by~\eqref{eq:scales}.

Suppose $U\ge1/4$. Then $\phi(J)=[U/4,U]$. The points in this interval within $3\tau_0$ of an integer have total length $O(\tau_0(U+1))$, hence relative measure $O(\tau_0)$, uniformly for $U\ge1/4$. Since $du/dz$ varies by at most a factor $2$ for $z=\phi(u)$ in this range, their preimage in $J$ also has relative measure $O(\tau_0)$. Any complete period containing a point outside this preimage has $G(\phi(u))=1$ throughout: the distance from an integer starts above $3\tau_0$ and changes by less than $\tau_0$. Call these periods \emph{good}. Their total length is at least
\[
 \Bigl(1-C\tau_0-\frac{4AQ}n\Bigr)|J|\ge\frac12|J|
\]
once $\tau_0$ is small and $n$ is large. When $G=1$, every complete period is good.

On each good period, multiplication by $G(\phi(u))$ leaves the integral in~\eqref{eq:period-mass} unchanged. Summing over these periods gives
\[
 \int_J H(u,\theta u)G(\phi(u))\,du\gg\delta|J|\asymp\frac{\delta n}{Q}.
\]

\step{Applying the prime comparison.}
Set $W(x,y,\zeta)=H(x,y)G(\zeta)$. By construction, positivity of $W(n/q,\theta n/q,M/q^2)$ would make $q$ divide exactly one term of the two blocks, to the first power, and nothing in $c!$. The square condition therefore forces
\[
 S(W)=\sum_{Q<q\le2Q}W\Bigl(\frac nq,\frac{\theta n}q,\frac M{q^2}\Bigr)=0.
\]
The derivative bounds for $f,g$, together with the fixed derivative bounds for $G$, give
\[
 \|\partial^\beta W\|_\infty\ll_\beta\delta^{-|\beta|}=(\log n)^{2|\beta|},
\]
so Lemma~\ref{lem:prime-comparison} applies with $C_W=2$ and $\Lambda=5$.

On the other hand, $|dv|/\log v\asymp Q^2(n\log Q)^{-1}\,du$ under $u=n/v$, so the corresponding integral satisfies
\[
 I(W)\gg\frac{Q^2}{n\log Q}\cdot\frac{\delta n}{Q}
 \gg_\varepsilon\frac Q{(\log n)^3}.
\]
Thus Lemma~\ref{lem:prime-comparison} gives $S(W)=I(W)+O(Q(\log n)^{-5})>0$ for large $n$, a contradiction.
\end{proof}

\section{Short blocks}\label{sec:short-blocks}

\begin{lemma}\label{lem:count}
Let $0<\varepsilon\le10^{-3}$. The number of integers $n\le X$ that have a prime factor $p>X^{1-\varepsilon}$ and satisfy
\[
 n!\,(n-1)!\,m!\,(m-\ell_2)!\,c!=\sq,\qquad n-1>m>m-\ell_2>c\ge0,\qquad 2\le\ell_2\le X^\varepsilon,
\]
for some $m,\ell_2,c$, is $O(X^{1/2+5\varepsilon}(\log X)^2)$.
\end{lemma}

\begin{proof}
Write $n=ap$, so that $a<X^\varepsilon$. By Lemma~\ref{lem:smallest}, $c\ll\ell_2\log X\le X^{\varepsilon+o(1)}$, so $p>\max(\ell_2,c,\sqrt n)$ for large $X$. As in the proof of Lemma~\ref{lem:transfer}, $m-j=tp$ for a unique $0\le j<\ell_2$ and some $1\le t<a$.

Fix $(a,t,\ell_2,j,c)$ and put
\begin{equation}\label{eq:remaining}
 f(z)=at\,c!\prod_{\substack{0\le i<\ell_2\\ i\ne j}}(z-i).
\end{equation}
Removing $p^2$ from~\eqref{eq:two-block} (with $\ell_1=1$) shows that $f(m)$ is a square. Each possible $m\le X$ determines $n=a(m-j)/t$.

For every prime $\sqrt X/2<q\le\sqrt X$, we have $q>\max(a,t,c,\ell_2)$ for large $X$. Thus $f$ has nonzero leading coefficient and $\ell_2-1$ distinct roots modulo $q$. By~\eqref{eq:weil}, the number $\rho(q)$ of residues for which $f$ is a square or zero satisfies
\[
 \rho(q)\le\frac q2+\frac{\ell_2-1+(\ell_2-2)\sqrt q}{2}\le\frac{3q}{4},
\]
uniformly in the parameters, since $\ell_2/\sqrt q\ll X^{\varepsilon-1/4}=o(1)$. Apply~\eqref{eq:sieve} to the possible values of $m$, with $\mathcal P$ consisting of the primes in this interval. As $(q-\rho(q))/\rho(q)\ge1/3$, the number of possible $m$ is
\[
 \ll X\Biggl(\sum_{\sqrt X/2<q\le\sqrt X}\frac{q-\rho(q)}{\rho(q)}\Biggr)^{-1}
 \ll\sqrt X\log X.
\]
There are $O(X^{5\varepsilon}\log X)$ parameter tuples, since $a,t,\ell_2,j\le X^\varepsilon$ and $c\ll X^\varepsilon\log X$. Summing gives $O(X^{1/2+5\varepsilon}(\log X)^2)$ values of $n$, which proves the stated bound.
\end{proof}

We now combine the two-block estimates of Sections~\ref{sec:two-blocks}--\ref{sec:short-blocks} with the three-factor bound~\eqref{eq:all-three}.

\begin{corollary}\label{cor:unit}
Let $0<\varepsilon\le10^{-3}$. Apart from at most $X^{1/2+6\varepsilon+o(1)}$ exceptions, every squarefree $n\in[X/2,X]$ with
\[
 P(n)>X^{1-\varepsilon/2}\qquad\text{and}\qquad P\bigl(s(n-1)\bigr)>X^{3\varepsilon}
\]
has no representation with fewer than five factorials, and each of its five-factor representations has the form
\[
 n!\,(n-1)!\,m!\,(m-1)!\,c!=\sq,\qquad n-1>m>m-1>c\ge1.
\]
\end{corollary}

\begin{proof}
Write $n=ap$ with $p=P(n)$, so $a<X^{\varepsilon/2}$. Consider a four- or five-factor representation, in the notation of~\eqref{eq:two-block}. Since $p>n^{1-\varepsilon/2}$, Lemma~\ref{lem:short} (with $\alpha=\varepsilon/2$) gives $\ell_1,\ell_2<X^\varepsilon$, and then Lemma~\ref{lem:smallest} gives $c\le X^{\varepsilon+o(1)}$. Thus $\max(\ell_1,c,a\ell_2)<X^{3\varepsilon}$, and Lemma~\ref{lem:transfer} forces $\ell_1=1$. By Lemma~\ref{lem:count}, $\ell_2\ge2$ occurs for at most $X^{1/2+6\varepsilon+o(1)}$ values of $n$. Otherwise $\ell_1=\ell_2=1$; a four-factor representation would then give $nm=\sq$ with $m<n$, which is impossible for squarefree $n$.

A two-factor representation would make $n$ a square (Section~\ref{subsec:main-term}), and by~\eqref{eq:all-three} only $O(\sqrt X)$ integers $n\le X$ have a three-factor representation.
\end{proof}

\section{Proof of Theorem~\ref{thm:main}}\label{sec:proof-main}

\subsection{The remaining five-factor representations}
Let $n=ap$ be as in Corollary~\ref{cor:unit}, and consider a five-factor representation $n!(n-1)!m!(m-1)!c!=\sq$. Its square condition is $nm\,c!=\sq$. By Lemma~\ref{lem:smallest}, $c\ll\log X<p$, so $p$ divides $m$, say $m=tp$ with $1\le t<a$. Removing $p^2$ we obtain $at\,c!=\sq$, that is, $t=s(ac!)v^2$ for some integer $v\ge1$. Such a $t<a$ exists if and only if $s(ac!)<a$. We are therefore led to the set
\[
 \mathcal R=\{a\ge1\ \text{squarefree}:\ s(ac!)\ge a\ \text{for all}\ c\ge0\}.
\]
If $a\in\mathcal R$, then $n$ has no five-factor representation. The condition is automatic for $c\le1$.

\begin{lemma}\label{lem:R}
The set $\mathcal R$ has positive lower density.
\end{lemma}

\begin{proof}
Let $B_0$ be large. The squarefree integers with no prime factor at most $B_0$ have density
\[
 \varrho_{B_0}=\frac6{\pi^2}\prod_{p\le B_0}\frac p{p+1}\asymp\frac1{\log B_0},
\]
and each of them satisfies $s(ac!)=a\,s(c!)\ge a$ for $c\le B_0$.

We now estimate how many of these integers fail the condition for some $c>B_0$. Put $s_c=s(c!)$. For squarefree $a$ we have $s(ac!)=as_c/\gcd(a,s_c)^2$, so $s(ac!)<a$ if and only if $\gcd(a,s_c)>\sqrt{s_c}$. In that case $a$ is divisible by a divisor $d>\sqrt{s_c}$ of $s_c$, so
\[
 \#\{a\le Z:a\ \text{squarefree},\ s(ac!)<a\}\le\sum_{\substack{d\mid s_c\\ d>\sqrt{s_c}}}\frac Zd\le Z\frac{\tau(s_c)}{\sqrt{s_c}}.
\]
The primes in $(c/2,c]$ show that $s_c\ge e^{(1/2+o(1))c}$, while $\tau(s_c)\le2^{\pi(c)}=e^{o(c)}$. Summing over $c>B_0$ gives a total loss of $O(Ze^{-\gamma B_0})$ from our initial set, for some absolute $\gamma>0$. Since $e^{-\gamma B_0}=o(\varrho_{B_0})$, the lower density of $\mathcal R$ is at least $\varrho_{B_0}/2$ for large $B_0$.
\end{proof}

This is a quantitative form of an observation of Erd\H{o}s and Graham~\cite[p.~354]{EG}.

\subsection{Supply of candidates}
For a squarefree $a\le X^\alpha$, the primes $p\in(X/(2a),X/a]$ give about $X/(a\log X)$ squarefree integers $n=ap\in(X/2,X]$. Summing over $a$ in a set of positive lower density yields $\gg X$ candidates. We first check that the condition on $n-1$ in Corollary~\ref{cor:unit} removes only a small proportion of them. Recall that an integer is $y$-smooth if all its prime factors are at most $y$.

\begin{lemma}\label{lem:smooth}
For all sufficiently small fixed $\varepsilon>0$, the number of $n\in[X/2,X]$ with $P(s(n-1))\le X^{3\varepsilon}$ is at most $CXe^{-1/(6\varepsilon)}+O(X^{3/4})$, where $C$ is absolute.
\end{lemma}

\begin{proof}
Write $n-1=s(n-1)u^2$. The number of $n$ with $u>X^{1/4}$ is at most $\sum_{u>X^{1/4}}X/u^2\ll X^{3/4}$. For the remaining $n$ we have $s(n-1)\ge X^{1/2}/3$ when $X$ is large.

Put $y=X^{3\varepsilon}$ and $\sigma=1/\log y$. Expanding into sums over squarefree divisors and using $\floor{X/d}\le X/d$ gives
\[
 \frac1X\sum_{v\le X}\prod_{\substack{p\mid v\\p\le y}}p^\sigma
 \le\prod_{p\le y}\Bigl(1+\frac{p^\sigma-1}{p}\Bigr)\ll1,
\]
because $p^\sigma-1\ll\sigma\log p$ for $p\le y$ and $\sum_{p\le y}(\log p)/p\ll\log y$. If $s(n-1)$ is $y$-smooth, its squarefreeness gives
\[
 \prod_{\substack{p\mid n-1\\p\le y}}p^\sigma\ge s(n-1)^\sigma
 \ge(X^{1/2}/3)^{1/\log y}=e^{1/(6\varepsilon)+o(1)}.
\]
The lemma follows from Markov's inequality.
\end{proof}

\begin{proposition}\label{prop:candidates}
Let $\mathcal A$ be a set of squarefree positive integers such that $\#(\mathcal A\cap[1,Y])\ge\kappa Y$ for some $\kappa>0$ and all large $Y$. Let $\varepsilon>0$ be sufficiently small in terms of $\kappa$, and put $\alpha=\varepsilon/2$. Then for all large $X$ there are $\gg_\kappa\alpha X$ squarefree integers $n=ap\in[X/2,X]$ with $a\in\mathcal A$, $X^{\alpha/2}\le a\le X^\alpha/2$, $p=P(n)>X^{1-\alpha}$ and $P(s(n-1))>X^{3\varepsilon}$.

These integers have no representation with fewer than five factorials, and every five-factor representation has two consecutive pairs as in Corollary~\ref{cor:unit}. If $\mathcal A\subseteq\mathcal R$, they have no five-factor representation either.
\end{proposition}

\begin{proof}
For each $a\in\mathcal A$ with $X^{\alpha/2}\le a\le X^\alpha/2$, let $p$ run over the primes in $(X/(2a),X/a]$. Then $p>X^{1-\alpha}$ is larger than $a$ and than $\sqrt X$, so $n=ap$ is squarefree, $P(n)=p$, and distinct pairs $(a,p)$ give distinct $n$. By the prime number theorem, the number of such $n$ is
\[
 \gg\frac X{\log X}\sum_{\substack{a\in\mathcal A\\ X^{\alpha/2}\le a\le X^\alpha/2}}\frac1a.
\]
To bound this sum from below, we use the positive lower density of $\mathcal A$. Partial summation gives
\[
 \sum_{\substack{a\in\mathcal A\\ Y\le a\le Y^2/2}}\frac1a\ge\int_Y^{Y^2/2}\frac{\#(\mathcal A\cap[1,u])}{u^2}\,du-1\ge\kappa\log Y-O(1).
\]
Taking $Y=X^{\alpha/2}$, we obtain $\gg\kappa\alpha X$ candidates.

By Lemma~\ref{lem:smooth}, at most $CXe^{-1/(6\varepsilon)}+O(X^{3/4})$ of them have $P(s(n-1))\le X^{3\varepsilon}$. Since $e^{-1/(6\varepsilon)}=o(\varepsilon)$, we can make this less than half the count above by choosing $\varepsilon$ sufficiently small in terms of $\kappa$. Corollary~\ref{cor:unit} removes a further $o(X)$ integers and gives the first assertion. The second follows from the discussion of $\mathcal R$ above.
\end{proof}

\subsection{Six factors}
By Lemma~\ref{lem:R}, Proposition~\ref{prop:candidates} applies with $\mathcal A=\mathcal R$ and yields $\gg X$ integers $n=ap\in[X/2,X]$ with no representation by fewer than six factorials. Since $a\ge X^{\alpha/2}$ and $p\ge X^{1-\alpha}$, the six arguments in~\eqref{eq:six} are distinct for large $X$, so $F(n)=6$. This proves the case $k=6$ of Theorem~\ref{thm:main}.

\subsection{Five factors}
If $d\,bc_0!=w^2$, then for $n=dz$ and $m=bz$,
\begin{equation}\label{eq:straus}
 n!\,(n-1)!\,m!\,(m-1)!\,c_0!=nm\,c_0!\bigl((n-1)!\,(m-1)!\bigr)^2=\bigl(wz\,(n-1)!\,(m-1)!\bigr)^2.
\end{equation}
Identities of this kind are due to Straus; see~\cite[Fact~7, pp.~346--347]{EG}.

\begin{proposition}\label{prop:five}
Let $d>1$ be squarefree, and suppose that $d\,bc_0!$ is a square for some integers $1\le b<d$ and $c_0\ge2$. Let $V\ge1$ be an integer coprime to $d$. Then for all large $X$ there are $\gg_{d,b,c_0,V}X$ squarefree $n\in[X/2,X]$ with $d\mid n$, $(n,V)=1$ and $F(n)=5$.
\end{proposition}

\begin{proof}
The set $\mathcal A$ of squarefree $a$ with $d\mid a$ and $(a,V)=1$ has density
\[
 \frac6{\pi^2}\prod_{q\mid d}\frac1{q+1}\prod_{q\mid V}\frac q{q+1}>0,
\]
so Proposition~\ref{prop:candidates} yields $\gg X$ integers $n=ap$ with no representation by fewer than five factorials. They satisfy $d\mid n$, and $(n,V)=1$ because $p>V$. Writing $n=dz$, identity~\eqref{eq:straus} is a five-factor representation of $n$, whose arguments $n>n-1>bz>bz-1>c_0$ are distinct for large $X$. Hence $F(n)=5$.
\end{proof}

For even $n$ we take $(d,b,c_0,V)=(2,1,2,1)$, which gives
\[
 (2z)!\,(2z-1)!\,z!\,(z-1)!\,2!=\bigl(2z\,(2z-1)!\,(z-1)!\bigr)^2,
\]
and for odd $n$ we take $(d,b,c_0,V)=(3,2,3,2)$, which gives
\[
 (3z)!\,(3z-1)!\,(2z)!\,(2z-1)!\,3!=\bigl(6z\,(3z-1)!\,(2z-1)!\bigr)^2.
\]
This completes the proof of Theorem~\ref{thm:main}.

\section{Further questions}\label{sec:questions}

By~\eqref{eq:nonsquarefree} and Theorem~\ref{thm:three},
\[
 \liminf_{X\to\infty}\frac{D_4(X)}X\ge1-\frac6{\pi^2}.
\]
If the squarefree integers with $F(n)=4$ had density zero, as Erd\H{o}s and Graham suggested~\cite[p.~346]{EG}, then $D_4(X)/X$ would tend to $1-6/\pi^2$. Our two-block argument requires a prime factor exceeding $X^{1-\alpha}$ for a small fixed $\alpha>0$, and does not decide this.

\begin{question}
Do the limits $\lim_{X\to\infty}D_k(X)/X$ exist for $k=4,5,6$? Can their values be expressed as Euler products or convergent series?
\end{question}

The proof of Lemma~\ref{lem:short} depends on the large prime factor of $n$ through the small denominator $A$ of $\theta$, which enters both estimates in~\eqref{eq:scales}.

\begin{question}
Does the conclusion of Lemma~\ref{lem:short} hold for all but $o(X)$ integers $n\le X$, without the assumption $P(n)>n^{1-\alpha}$?
\end{question}

\appendix
\section{Perfect powers from a bounded number of factorials}\label{app:affine}

At the end of their paper, Erd\H{o}s and Graham~\cite[p.~354]{EG} state that, for every $m\ge2$, infinitely many $m$-th powers can be written as products of a bounded number of factorials. They leave the proof for a later paper. B\'erczes, Dujella, Hajdu, Saradha and Tijdeman~\cite[p.~339]{BDHST} report that they could not find this continuation. In this appendix we give a construction with distinct factorials whose arguments are affine functions of a parameter. Throughout the appendix, $m$ denotes the exponent.

\begin{theorem}\label{thm:affine}
For every integer $m\ge2$ there are an integer $N\ge1$ and distinct affine functions $a_i(t)=\alpha_it+\beta_i$, $1\le i\le N$, with integer coefficients $\alpha_i>0$ and $\beta_i\ge0$, such that for every sufficiently large integer $t$ the numbers $a_i(t)$ are distinct and at least $2$, and $\prod_{i=1}^Na_i(t)!$ is an $m$-th power. Moreover one can take $m\mid N$ and
\[
 N\le\begin{cases}2m(m-1)\lceil e^m\rceil&\text{if $m$ is a prime power},\\ 2c_m(c_m+1)\lceil e^{c_m+1}\rceil&\text{in general},\end{cases}
 \qquad c_m=\lceil m\log m\rceil.
\]
In particular $N\le\exp(O(m\log m))$.
\end{theorem}

Theorem~2.1 of~\cite{BDHST} concerns products in which factorials may be repeated, and~\cite[Theorem~4.2]{BDHST} gives, for infinitely many $n$, cubes that are products of fewer than $\frac32\log n$ distinct factorials with largest argument $n$. Theorem~\ref{thm:affine} gives a number of factors depending only on the exponent, along an arithmetic progression of largest arguments (Corollary~\ref{cor:ap}). Related constructions for products of disjoint blocks of consecutive integers were studied by Ska{\l}ba~\cite{S03,S19}. Expressing a block as a factorial ratio and clearing the denominator modulo $m$-th powers generally repeats the denominator factorial $m-1$ times; the construction below preserves distinctness by working directly with individual factorials.

For a finite abelian group $G$, the Davenport constant $\mathsf D(G)$ is the least integer $D$ such that every sequence of $D$ elements of $G$, repetitions allowed, has a nonempty subsequence with sum zero. Let $C_m=\Z/m\Z$. By Olson's theorem~\cite{O},
\[
 \mathsf D(C_m^r)=1+r(m-1)\qquad\text{if $m$ is a prime power},
\]
and the bound of van Emde Boas and Kruyswijk~\cite{EK} (see also Meshulam~\cite{M}) gives, for $m\ge3$,
\[
 \mathsf D(C_m^r)\le\floor{m\bigl(1+(r-1)\log m\bigr)}+1\le c_mr.
\]
The last inequality holds because the quantity inside the floor is strictly less than the integer $rc_m$. For $m=2$, Olson's formula gives $\mathsf D(C_2^r)=r+1\le2r=c_2r$. Thus $\mathsf D(C_m^r)\le c_mr$ for every $m\ge2$.

\begin{proof}[Proof of Theorem~\ref{thm:affine}]
We express the factorials as monomials in a fixed finite set of formal generators. The Davenport constant then supplies a nonempty subfamily whose exponent vectors sum to zero modulo $m$.

Let $R,K\ge1$ and $L_R=\lcm(1,\dots,R)$. For $1\le\ell\le R$ and $0\le j\le\floor{K/\ell}$ consider
\[
 F_{\ell,j}(t)=\Bigl(\frac{L_Rt}\ell+j\Bigr)!.
\]
Call the terms with a given $\ell$ the $\ell$-th row. For $j\ge1$, passing from $F_{\ell,j-1}$ to $F_{\ell,j}$ multiplies the factorial by $(L_Rt+\ell j)/\ell$. Since $\ell j\le K$, all rows use linear factors from the same list $L_Rt+1,\dots,L_Rt+K$. With the generators
\[
 \Phi_\ell(t)=\Bigl(\frac{L_Rt}\ell\Bigr)!\quad(1\le\ell\le R),\qquad Y_i(t)=L_Rt+i\quad(1\le i\le K),
\]
and the primes $p\le R$, we have
\[
 F_{\ell,j}(t)=\Phi_\ell(t)\,\ell^{-j}\prod_{h=1}^jY_{\ell h}(t),\qquad\ell^{-j}=\prod_{p\le R}p^{-jv_p(\ell)},
\]
where $v_p(\ell)$ is the exponent of $p$ in $\ell$ and $\ell h\le K$. Thus each $F_{\ell,j}$ is a monomial in $r=R+K+\pi(R)$ generators, with an exponent vector in $\Z^r$ (possibly with negative entries) that does not depend on $t$. Reducing these vectors modulo $m$ gives a sequence of
\[
 M_0=\sum_{\ell=1}^R\Bigl(\floor{\frac K\ell}+1\Bigr)>K\sum_{\ell=1}^R\frac1\ell>K\log R,
\]
elements of $C_m^r$.

Suppose that $M_0\ge\mathsf D(C_m^r)$. Then some nonempty subsequence of at most $\mathsf D(C_m^r)$ terms has sum zero. The selection is independent of $t$. For the corresponding factorials, every generator occurs with total exponent divisible by $m$. Thus for every positive integer $t$ their product is the $m$-th power of a positive rational number. Since the product is an integer, this rational number is an integer too.

Let $N$ be the number of selected factorials. The generator $\Phi_\ell$ occurs exactly in the terms of the $\ell$-th row, with exponent $1$, so the number of selected terms in each row is divisible by $m$; in particular $m\mid N$.

Arguments in the same row are distinct. Arguments in different rows have different slopes $L_R/\ell$, whose differences are at least $1$, and shifts in $[0,K]$; so they are distinct for $t>K$. All arguments are at least $L_Rt/R\ge2$ for large $t$.

The number of available factorials exceeds $K\log R$, whereas the number of generators is $K+R+\pi(R)$. We choose $R$ first to make $\log R$ large enough, and then $K$ to ensure $M_0\ge\mathsf D(C_m^r)$. If $m$ is a prime power, let $R=\lceil e^m\rceil$, $R_0=R+\pi(R)$ and $K=(m-1)R_0$. Then $r=mR_0$ and $M_0>K\log R\ge m(m-1)R_0=r(m-1)$, so $M_0\ge\mathsf D(C_m^r)$ and $N\le\mathsf D(C_m^r)=1+m(m-1)R_0$. Since $m\mid N$, we get $N\le m(m-1)R_0\le2m(m-1)\lceil e^m\rceil$.

For general $m$, let $R=\lceil e^{c_m+1}\rceil$, $R_0=R+\pi(R)$ and $K=c_mR_0$. Then $r=(c_m+1)R_0$ and $M_0>K\log R\ge c_m(c_m+1)R_0=c_mr\ge\mathsf D(C_m^r)$, which gives $N\le c_m(c_m+1)R_0\le2c_m(c_m+1)\lceil e^{c_m+1}\rceil$.
\end{proof}

\begin{corollary}\label{cor:ap}
For every $m\ge2$ there are $N\ge1$, $q_0\ge1$, a residue class $b$ modulo $q_0$ and $\lambda>0$ such that every sufficiently large $n\equiv b\pmod{q_0}$ admits a representation
\[
 a_1!\cdots a_N!=y^m,\qquad\lambda n\le a_1<\cdots<a_N=n,
\]
with $N$ bounded as in Theorem~\ref{thm:affine}.
\end{corollary}

\begin{proof}
For large $t$ the order of the functions $a_i(t)$ in Theorem~\ref{thm:affine} is fixed. The largest has the form $q_0t+b$, and takes every sufficiently large value $n\equiv b\pmod{q_0}$, while the ratio of the smallest to the largest tends to $\min_i\alpha_i/\max_i\alpha_i>0$.
\end{proof}

\begin{remark}
If only a specified set of rows is used, $L_R$ can be replaced by the least common multiple of their denominators. The construction is finite and effective.
\end{remark}

\section*{Acknowledgements}
Large language models (ChatGPT and Claude) were used in developing and checking arguments, locating references, and editing the exposition. The author is responsible for all content.

\end{document}